\documentclass[10pt]{amsart}
\usepackage{amssymb}
\usepackage{tikz-cd}
\usepackage{amsrefs}
\usepackage{xfrac}
\usepackage{thmtools}
\usepackage{thm-restate}

\title{Continuity of measurable cocycles}

\author {Christian Rosendal}
\address{Department of Mathematics\\University of Maryland\\4176 Campus Drive - William E. Kirwan Hall\\College Park, MD 20742-4015\\USA}
\email{rosendal@umd.edu}
\urladdr{sites.google.com/view/christian-rosendal/}

\newcommand{\forkindep}[1][]{\mathop{\mathop{\vcenter{\hbox{\oalign{\noalign{\kern-.3ex}
\hfil$\vert$\hfil\cr\noalign{\kern-.7ex}$\smile$\cr\noalign{\kern-.3ex}}}}}\displaylimits_{#1}}}

\newcommand{\maths}[1]{\[\begin{split}{#1}\end{split}\]}

\newcommand{\maps}[1]{\mathop{\overset{#1}\longrightarrow}}

\newcommand{\acts}[1]{\mathop{\overset{#1}\curvearrowright}}

\newcommand{\Mgd}[2]{\big\{  {#1}\;\big|\; {#2} \big\} }

\usetikzlibrary{decorations.markings}
\tikzset{negated/.style={
        decoration={markings,
            mark= at position 0.5 with {
                \node[transform shape] (tempnode) {$\backslash$};
            }
        },
        postaction={decorate}
    }
}

\newcommand {\R}{\mathbb R}

\newcommand{\tom} {\emptyset}

\newcommand{\saa}{\Rightarrow}
\newcommand{\equi}{\Leftrightarrow}

\newcommand {\del}{ \; \big| \;}

\newcommand {\ku} {\mathcal}

\newcommand{\ov}{\overline}
\newcommand{\inv}{^{-1}}
\newcommand {\e} {\exists}
\renewcommand {\a} {\forall}

\theoremstyle{plain}
\newtheorem{thm}{Theorem}[section]
\newtheorem*{theorem*}{Theorem}
\newtheorem{cor}[thm]{Corollary}
\newtheorem{lemme}[thm]{Lemma}

\theoremstyle{definition}

\newtheorem{rem}[thm]{Remark}

\definecolor{groen}{rgb}{0,0.5,.7}
\definecolor{gul}{rgb}{0.94,0.8,0}
\definecolor{blaa}{rgb}{0.16,0,0.6}
\definecolor{roed}{rgb}{1,0,0}

\makeindex

\thanks{The author was partially supported by the U.S. National Science Foundation under Grant Number DMS-2246986. Thanks are also due to Uri Bader and  Rémi Barritault for interesting comments and helpful conversations.}

\begin{document}

\begin{abstract}
Suppose $G\curvearrowright X$ is a Polish group action, $H$ is a Polish group and  $G\times X\maps\psi H$ is a cocycle that is continuous in the second variable. If $\psi$ is either Baire measurable or is $\lambda\times \mu$-measurable with respect to a Haar measure $\lambda$ on $G$ and a fully supported $\sigma$-finite Borel measure $\mu$ on $X$, then $\psi$ is jointly continuous.
\end{abstract}

\maketitle


\section{Introduction}
 This paper establishes some new instances of topological rigidity of cocycles in the context of Polish groups and their actions. For the setup, let us recall that by a {\em Polish group action} we understand a continuous action $G\curvearrowright X$  by a Polish group $G$ on a Polish space $X$, that is, such that the action map $G\times X\to X$ is jointly continuous. Also, if $G\curvearrowright X$ is any action by a group $G$ on a set $X$, a  group valued {\em cocycle} is a map $G\times X\maps\psi H$ with values in another group $H$ such that
$$
\psi(gf,x)=\psi(g,fx)\cdot\psi(f,x)
$$
for all $g,f\in G$ and $x\in X$. For example, if $X\maps \phi H$ is any function, then the differential
$$
d\phi(g,x)=\phi(gx)\inv\phi(x)
$$
defines a so called {\em trivial} cocycle, but general cocycles need not be that simple. The principal example of a nontrivial cocycle comes from calculus. Namely, if we let ${\sf Diff}^1(\R)$ denote the group of $C^1$-diffeomorphisms of $\R$, then the chain rule simply expresses that the map 
$$
(f,x)\in {\sf Diff}^1(\R)\times \R\mapsto f'(x)\in \R^\times
$$
is a cocycle for the tautological action of ${\sf Diff}^1(\R)$ on $\R$. 

There are several results in the literature on the continuity of cocycles under various notions of measurability (see, for example, \cite{becker}). Our main result, however, is most closely related to a recent result by T. Meyerovitch and O. N. Solan \cite[Theorem 1.6]{meyerovitch}. 

To state this precisely, suppose $G\times X\maps\psi H$ is a cocycle for a Polish group action $G\curvearrowright X$ with values in a topological group $H$. Then, for fixed $g\in G$ and $x\in X$, we may define the section maps $X\maps {\psi_g}H$ and $G\maps{\psi_x}H$ by $\psi_g(z)=\psi(g,z)$ and $\psi_x(f)=\psi(f,x)$ respectively. We say that  $\psi$ is  {\em continuous in the second variable} if the maps $X\maps {\psi_g}H$ are continuous for all $g\in G$. Similarly for the other variable or for notions of measurability.
Let us remark that the relationship between joint and separate measurability is in general more subtle than that of continuity. Indeed, if a map $G\times X\maps\psi H$ is continuous, then it is automatically continuous in each variable. However, if $G\times X\maps\psi H$ is only assumed to be Baire measurable, then it need not be Baire measurable in any of the two variables. Similarly, if $G\times X\maps\psi H$ is just assumed to be measurable with respect to a product $\lambda\times \mu$ of $\sigma$-finite Borel measures on $G$ and $X$, then it need not be measurable in any of the two variables.

In this language, \cite[Theorem 1.6]{meyerovitch} states that, if $G$, $H$ and $X$ are locally compact Polish,   $G\times X\maps\psi H$ is a cocycle for a continuous action $G\curvearrowright X$, $\psi$ is assumed to be continuous in the second variable and $\psi_x$ is Haar measurable for every $x\in X$, then $\psi$ is continuous.  We extend this result in multiple directions.
\begin{thm}\label{thm:cocyclemain}
Let $G\curvearrowright X$ be a  Polish group action and suppose that $G\times X\maps\psi H$ is a cocycle with values in a Polish group $H$. Assume also that $\psi$ is continuous in the second variable. Then $\psi$ is continuous in each of the following cases. 
\begin{enumerate}
\item $\psi$ is Baire measurable,
\item $\psi_x$ is Baire measurable  for a dense set of $x\in X$,
\item $G$ is locally compact and $\psi_x$ is Haar measurable  for a dense set of $x\in X$,
\item $G$ is locally compact with Haar measure $\lambda$ and there is a fully supported $\sigma$-finite Borel measure $\mu$ on $X$ such that $\psi$ is $\lambda\times \mu$-measurable.
\end{enumerate}
\end{thm}
Observe, in particular, that case (3) is a direct extension of \cite[Theorem 1.6]{meyerovitch} discarding the hypothesis of local compactness of $H$ and $X$ and weakening the measurability assumption on $\psi$. Note also that the measure $\mu$ in (4) is not assumed to be invariant or even quasi-invariant under the $G$-action.

\begin{rem}\label{remark}
For completeness and later use, we mention that the assumption of continuity in the second variable can be slightly weakened. In fact, it suffices to assume that $\psi_g$ is continuous for all $g$ belonging to a set $\Sigma$ that generates $G$. For, if $\psi_g$ and $\psi_f$ are continuous, so are $\psi_{gf}=\psi_g(f\,\cdot\,)\psi_f(\cdot)$ and $\psi_{g\inv}=\big(\psi_g(g\inv\,\cdot\,)\big)\inv$. Thus, for example, it is enough to assume that $\psi_g$ is continuous for a comeagre set of $g\in G$ or, in the case of locally compact groups, a conull set of $g\in G$.
\end{rem}


Even in the case of trivial cocycles, Theorem \ref{thm:cocyclemain} provides nontrivial information, assuming that we restrict ourselves to transitive Polish group actions $G\acts{}X$. Observe first that, if $X\maps\phi H$ and $g\in G$, we may define the {\em directional differential} in the direction of $g$ to be the section map $d_g\phi=(d\phi)_g$. Since $d_g\phi$ is itself a function from $X$ to $H$, this means that we can iterate this procedure to get higher order directional differentials. Thus, for example,
\maths{
d_g\phi(x)&=\phi(gx)\inv\phi(x),\\
d_fd_g\phi(x)&=\phi(fx)\inv\phi(gfx)\phi(gx)\inv\phi(x),\\
d_kd_fd_g\phi(x)&=\phi(kx)\inv\phi(gkx)\phi(gfkx)\inv\phi(fkx)  \phi(fx)\inv\phi(gfx)\phi(gx)\inv\phi(x).
}

\begin{thm}\label{thm:differentials}
Suppose $G\acts{}X$ is a transitive Polish group action and that $X\maps\phi H$ is a Baire measurable function with values in a Polish group $H$.
Then the following conditions are equivalent.
\begin{enumerate}
\item $\phi$ is continuous,
\item $d\phi$ is continuous,
\item there is a comeagre set of $g\in G$ such that $d_g\phi$ is continuous,
\item for every proper subgroup $F<G$, there is some $g\in G\setminus F$ so that $d_g\phi$ is continuous,
\item for any infinite sequence $g_1,g_2,\ldots \in G$, there is some $n$ so that
$$
d_{g_n}d_{g_{n-1}}\cdots d_{g_1}\phi
$$
is continuous.
\end{enumerate}
\end{thm}


Similarly to \cite[Theorem 1.6]{meyerovitch}, Theorem \ref{thm:differentials} has consequences for the notion of  polynomial maps between groups, which is  due to A. Leibman \cite{leibman} and is especially useful in the setting of nilpotent groups, but which applies to all groups in general.

Suppose $G\maps\phi H$ is a map between two groups. Viewing $G$ as  a $G$-space under the action $G\curvearrowright G$ by left multiplication, we may define the differential $G\times G \maps{d\phi}H$ as before. 
For $k\geqslant 0$,  a map $G\maps\phi H$ between two groups is said to be a {\em polynomial of degree $\leqslant k$} in case
$$
d_{g_k}\cdots d_{g_0}\phi\;\equiv\; 1_H
$$
for all $g_0,\ldots, g_k\in G$. More generally, we may  define $G\maps\phi H$ to be a {\em polynomial} of potentially  transfinite degree provided that, for all $g_0,g_1,g_2,\ldots\in G$, there is a $k$ such that  
$$
d_{g_k}\cdots d_{g_0}\phi\;\equiv\; 1_H.
$$ 
It follows immediately from the definition that $\phi$ has degree $\leqslant 0$ if and only if $\phi$ is a constant map. Similarly, $\phi$ has degree $\leqslant 1$ if and only if $\phi(1)\phi(\,\cdot\,)\inv$ is a group homomorphism.

Because constant maps $G\to H$ are obviously continuous, the following is an immediate consequence of Theorem \ref{thm:differentials}.
\begin{cor}
Every Baire measurable polynomial  $G\maps\phi H$  between Polish groups is continuous.
\end{cor}


\section{Measurable cocycles}
In this section we prove Theorem \ref{thm:cocyclemain}. Observe first that, if  $G\times X\maps\psi H$ is a cocycle associated with a group action  $G\curvearrowright X$, then by applying the cocycle equation to $g=f=1$ we have that $\psi(1,x)=1$ for all $x\in X$.

\begin{lemme}\label{lem:measurability}
Let $(\Omega,\ku S)$ be a measurable space and $X$, $Y$ be separable metrisable topological spaces. Assume also that $\Omega \times X\maps \psi Y$ is a map such that $\psi_\omega\colon X\to Y$ is continuous for all $\omega\in \Omega$, whereas $\psi_x\colon \Omega\to Y$ is measurable for a dense set of $x\in X$. Then $\psi_x\colon \Omega\to Y$ is measurable for all $x\in X$.
\end{lemme}

\begin{proof}
Let $\ku X\subseteq X$ be a countable dense subset such that $\psi_z$ is $\ku S$-measurable for all $z\in \ku X$.
Suppose also $W\subseteq Y$ is a given non-empty open set and find a countable collection $\ku V$ of open subsets $V\subseteq W$ so that $W=\bigcup_{V\in \ku V}V=\bigcup_{V\in \ku V}\ov V$. Also, for a given point $x\in X$, let $\ku N$ be a countable neighbourhood basis for $x$. 

Then, for any $\omega\in \Omega$, we have by the continuity in the second variable that
\maths{
\omega\in \psi_x\inv(W)
&\equi \psi(\omega,x)\in W\\
&\equi \e V\in \ku V\;\; \psi(\omega,x)\in V\\
&\equi \e V\in \ku V\;\; \e U\in \ku N\;\; \psi\big[\{\omega\}\times U\big]\subseteq  V\\
&\saa \e V\in \ku V\; \;\e U\in \ku N\;\; \a z\in U\cap \ku X\; \;\psi(\omega,z)\in  V\\
&\equi \omega\in \bigcup_{V\in \ku V}\;\bigcup_{U\in \ku N}\;\bigcap_{z\in U\cap  \ku X}\; \psi_z\inv(V) \\
&\saa \e V\in \ku V\; \;\e U\in \ku N\;\; \psi\big[\{\omega\}\times U\big]\subseteq  \ov V\\
&\saa \e V\in \ku V\; \;\psi(\omega,x)\in \ov V\\
&\saa \omega\in \psi_x\inv(W).
}
It thus follows that
$$
\psi_x\inv(W)= \bigcup_{V\in \ku V}\;\bigcup_{U\in \ku N}\;\bigcap_{z\in U\cap  \ku X}\; \psi_z\inv(V)
$$
is $\ku S$-measurable.
\end{proof}

\begin{lemme}\label{lem:cocycle}
Let $G\curvearrowright X$ be a continuous action by a topological group $G$ on a topological space $X$ and suppose that $G\times X\maps\psi H$ is a cocycle with values in a topological group $H$ such that $\psi$ is continuous in the second variable. Assume also that, for some $x\in X$, the map $\psi_x$ is continuous at a single point in $G$. Then $\psi_x$ is continuous at every point of $G$.
\end{lemme}

\begin{proof}
Suppose $\psi_x$ is continuous at some point $f\in G$ and let $g\in G$ be any given point. To see that $\psi_x$ is also continuous at $gf$, suppose that $k_i\to gf$ is a convergent net.  Then $g\inv k_i\to f$ and hence
\maths{
\lim_i{\psi}(k_i,x)
&=\lim_i{\psi}(gg\inv k_i,x)\\
&=\lim_i\Big({\psi}(g ,g\inv k_ix)\;\cdot\; {\psi}(g\inv k_i, x)\Big)\\
&=\lim_i{\psi}(g ,g\inv k_ix)\;\cdot\;\lim_i {\psi}(g\inv k_i, x)\\
&={\psi}(g, fx)\cdot {\psi}(f ,x)\\
&={\psi}(gf,x),
}
showing continuity of $\psi_x$ at $gf$.
\end{proof}

\begin{lemme}\label{lem:coc}
Let $G\curvearrowright X$ be a continuous action by a topological group $G$ on a topological space $X$ and suppose that $G\times X\maps\psi H$ is a cocycle with values in a Polish group $H$ such that $\psi$ is continuous in the second variable. Then, if $\psi$ is continuous at some point $(f,x)$, $\psi$ is in fact continuous at every point of $G\times Gx$.
\end{lemme}

\begin{proof}
We first show that, if the cocycle $\psi$ is continuous at a point $(f,x)$, then it is also continuous at all points $(gf,x)$ for $g\in G$. To see this, suppose that $k_i\to gf$ and $x_i\to x$ are convergent nets.  Then $g\inv k_i\to f$ and hence
\maths{
\lim_i{\psi}(k_i,x_i)
&=\lim_i{\psi}(gg\inv k_i,x_i)\\
&=\lim_i\Big({\psi}(g ,g\inv k_ix_i)\;\cdot\; {\psi}(g\inv k_i, x_i)\Big)\\
&=\lim_i{\psi}(g ,g\inv k_ix_i)\;\cdot\;\lim_i {\psi}(g\inv k_i, x_i)\\
&={\psi}(g, fx)\cdot {\psi}(f ,x)\\
&={\psi}(gf,x).
}

We next show that, if the cocycle $\psi$ is continuous at a point $(f,x)$, then it is also continuous at all points $(f,gx)$ for $g\in G$. To see this, suppose that $f_i\to f$ and $y_i\to gx$.  Then $f_ig\to fg$ and $g\inv y_i\to x$. Because  $\psi$ is continuous at $(fg,x)$, it follows that\maths{
\lim_i{\psi}(f_i,y_i)
&=\lim_i{\psi}(f_igg\inv, y_i)\\
&=\lim_i\Big({\psi}(f_ig ,g\inv y_i)\;\cdot\; {\psi}(g\inv , y_i)\Big)\\
&=\lim_i{\psi}(f_ig ,g\inv y_i)\;\cdot\; \lim_i{\psi}(g\inv , y_i)\\
&={\psi}(fg, x)\cdot {\psi}(g\inv ,gx)\\
&={\psi}(fg, g\inv gx)\cdot {\psi}(g\inv ,gx)\\
&={\psi}(fgg\inv ,gx)\\
&={\psi}(f,gx).
}
The lemma now follows from the conjunction of these two facts.
\end{proof}

\begin{proof}[Proof of Theorem  \ref{thm:cocyclemain} (2)]
Recall that we are given a Polish group action $G\curvearrowright X$  and a cocycle $G\times X\maps\psi H$  with values in a Polish group $H$ such that $\psi$ is continuous in the second variable. Furthermore, $\psi_x$ is assumed to be Baire measurable for a dense set of $x\in X$. Thus, by Lemma \ref{lem:measurability}, $\psi_x$ will be Baire measurable for all $x\in X$.

We first show that, for any $x\in  X$, the function $\psi_x$ has a point of continuity in  $G$. To see this, note that, for any fixed $g\in G$, 
$\psi_x$ is continuous at $g$ if and only if, for every open identity neighbourhood $W\subseteq H$, there is some $h\in H$ such that
$$
g\in {\sf int}\Big(\psi_x\inv(hW^2)\Big).
$$
Thus, by the Baire category theorem and  first countability of $H$, it suffices to show that, for every open identity neighbourhood $W\subseteq H$,  the open set 
$$
\bigcup_{h\in H}{\sf int}\Big(\psi_x\inv(hW^2)\Big)
$$
is dense in $G$. 

So let $W$ be given and fix some countable dense subset $\ku H\subseteq H$. Assume also that $O\subseteq G$ is a given non-empty open set and find non-empty open subsets $U_1,U_2\subseteq G$ so that $U_1U_2\subseteq O$. 
Then
$$
U_2\subseteq \bigcup_{h\in \ku H}\psi_x\inv(hW)
$$
and hence there must be some $h_1\in \ku H$   such that $U_2\cap \psi_x\inv(h_1W)$ is nonmeagre. 

Choose now some open identity neighbourhood $W'\subseteq H$ so that $\ov{W'}h_1\subseteq h_1W$ and fix an element $f\in U_2\cap \psi_x\inv(h_1W)$ so that $U_2\cap \psi_x\inv(h_1W)$ is comeagre in a neighbourhood of $f$.
As above, we may then find some $h_2\in \ku H$   so that $U_1\cap \psi_{fx}\inv(h_2W')$ is nonmeagre.

Fix  a countable neighbourhood basis $\ku V$ for $f$ consisting of open sets and a countable dense subset $\ku G$ of $G$. We claim that
\maths{
\psi_{fx}\inv(h_2W')
&\;\subseteq\;      \bigcup_{V\in \ku V}\Mgd{g\in G}{\a k\in V \; \psi(g,kx)\in h_2{W'}}.
}
Indeed, suppose that $g\in \psi_{fx}\inv(h_2W')$ is given. Then $\psi(g,fx)\in h_2W'$ and so, because $\psi$ is continuous in the second variable, there is some $V\in \ku V$ such that $\psi(g,kx)\in h_2W'$ for all $k\in V$, which proves the claim. 
Because  $U_1\cap \psi_{fx}\inv(h_2W')$ is nonmeagre, it  follows that there is some $V\in \ku V$ such that 
$$
U_1\cap \Mgd{g\in G}{ \a k\in V \;\psi(g,kx)\in h_2W'}
$$
is nonmeagre.

Observe also that
\maths{
\Mgd{g\in G}{\a k\in V \; \psi(g,kx)\in h_2{W'}}
&\;\subseteq\; 
\bigcap_{k\in  V\cap \ku G} \psi_{kx}\inv(h_2W')\\
&\;\subseteq\;   
\Mgd{g\in G}{\a k\in V \; \psi(g,kx)\in h_2\ov{W'}}.
}
Indeed, the first inclusion is immediate and to verify the last inclusion assume $\psi(g,kx)\in h_2W'$ for all $k\in V\cap \ku G$.
Then every $k\in V$ is the limit of some sequence $(k_n)\subseteq V\cap \ku G$ and therefore
$$
\psi(g,kx)=\lim_n\psi(g,k_nx)\in \ov{h_2W'}=h_2\ov {W'}
$$
as claimed. We thus conclude that 
$$
U_1\cap  \bigcap_{k\in  V\cap \ku G} \psi_{kx}\inv(h_2W')
$$
is nonmeagre.

Observe that the two factors in the product 
$$
A=\Big(U_1\cap \bigcap_{k\in  V\cap \ku G} \psi_{kx}\inv(h_2W')\Big)
\cdot 
\Big(U_2\cap V\cap \psi_x\inv(h_1W)\Big)
$$
are both nonmeagre sets with the Baire property. By the Pettis theorem \cite{pettis} (see \cite[Lemma 2.1]{autom} for  the exact statement needed), the product $A$  therefore has non-empty interior. Moreover, for 
$$
g\;\in\; U_1\cap \bigcap_{k\in  V\cap \ku G} \psi_{kx}\inv(h_2W') 
$$
and 
$$
k\in U_2\cap V\cap \psi_x\inv(h_1W),
$$
we have
\maths{
\psi(gk,x)
&=\psi(g,kx)\cdot\psi(k,x)\\
&\in h_2\ov{W'} \cdot h_1W\\
&\subseteq h_2h_1W^2.
}
In other words, $A\subseteq \psi_x\inv \big(h_2h_1W^2\big)$. As also  $\tom\neq {\sf int}\, A\subseteq A\subseteq U_1U_2\subseteq O$, we see that $\bigcup_{h\in H}{\sf int}\Big(\psi_x\inv(hW^2)\Big)$
is dense in $G$.

We have shown that, for every $x\in X$, the map $\psi_x$ is continuous at some point of $G$. Therefore, by Lemma \ref{lem:cocycle}, $\psi_x$ is continuous at every point of $G$. By the assumptions of the theorem, $G\times X\maps\psi H$ is separately continuous. 

We may now apply \cite[Theorem 8.51]{kechris} to conclude that there is a comeagre subset $Z\subseteq G\times X$  such that the section 
$$
Z^x=\Mgd{g\in G}{ (g,x)\in Z}
$$
is comeagre for every $x\in X$ and so that $\psi$ is continuous at every point of $Z$. In particular, this shows that, for all $x\in X$ there is some $g\in G$ such that $\psi$ is continuous at $(g,x)$. By Lemma \ref{lem:coc}, it follows that $\psi$ is continuous at all points of $G\times X$.
\end{proof}

\begin{proof}[Proof of Theorem  \ref{thm:cocyclemain} (3)]
The proof of this is very similar to that of Theorem  \ref{thm:cocyclemain} (2), so we will keep the same notation and just point out the small changes that are needed. We begin by fixing a left-invariant Haar measure $\lambda$ on $G$.

As before, we must show that the open set $\bigcup_{h\in H}{\sf int}\Big(\psi_x\inv(hW^2)\Big)$ intersects some given non-empty open set $O\subseteq G$. We choose $U_1U_2\subseteq O$ as before and find $h_1\in \ku H$   so that $U_2\cap \psi_x\inv(h_1W)$ is nonnull. 

Again we find $W'$ satisfying $\ov{W'}h_1\subseteq h_1W$. Define now 
$$
D=\bigcup\Mgd{U\subseteq X \text{ open }}{\lambda\big(U\cap U_2\cap \psi_x\inv(h_1W)\big)=0}.
$$
Then $D$ is open and $\lambda\big(D\cap U_2\cap \psi_x\inv(h_1W)\big)=0$. Because $U_2\cap \psi_x\inv(h_1W)$ is non-null, we may therefore find some $f\in \big(U_2\cap \psi_x\inv(h_1W)\big)\setminus D$, which means that $U_2\cap \psi_x\inv(h_1W)$ has non-null intersection with every neighbourhood of $f$. 

Choose then  $h_2\in \ku H$   so that $U_1\cap \psi_{fx}\inv(h_2W')$ is non-null and find an open neighbourhood $V$ of $f$ so that
$$
U_1\cap  \bigcap_{k\in  V\cap \ku G} \psi_{kx}\inv(h_2W')
$$
is non-null. Then  the two factors in the product 
$$
\Big(U_1\cap \bigcap_{k\in  V\cap \ku G} \psi_{kx}\inv(h_2W')\Big) 
\cdot 
\Big(U_2\cap \psi_x\inv(h_1W)\cap V\Big) \;\;\subseteq\;\; \psi_x\inv \big(h_2h_1W^2\big)
$$
are both non-null Haar measurable sets. It follows therefore that the product  has non-empty interior (see,  e.g., \cite[Theorem 2.3]{autom}). The rest of the proof remains unaltered.
\end{proof}

\begin{proof}[Proof of Theorem  \ref{thm:cocyclemain} (1)]
Let $\ku W$ be a countable basis for the topology on $H$. Then, for every $W\in \ku W$, the inverse image  $\psi\inv(W)$ has the Baire property in $G\times X$ and therefore, by the  Kuratowski--Ulam  theorem \cite[Theorem 8.41]{kechris}, there is a comeagre subset $Z_W\subseteq X$ such that 
$$
\psi\inv(W)_z=\Mgd{g\in G}{(g,z)\in \psi\inv(W)}=\psi_z\inv(W)
$$
has the Baire property in $G$ for all $z\in Z_W$. Thus, for all $z$ belonging to the comeagre set $\bigcap_{W\in \ku W}Z_W$ and all $W\in \ku W$, $\psi_z\inv(W)$ has the Baire property, whereby  $\psi_z$ is Baire measurable. The result now follows from Theorem  \ref{thm:cocyclemain} (2).
\end{proof}

\begin{proof}[Proof of Theorem  \ref{thm:cocyclemain} (4)]
Let $\ku W$ be a countable basis for the topology on $H$. Then, for every $W\in \ku W$, the inverse image  $\psi\inv(W)$ is a $\lambda\times \mu$-measurable subset of $G\times X$. Because both $\lambda$ and $\mu$ are $\sigma$-finite Borel measures, it follows from  the  Fubini--Tonelli  theorem that there is a $\mu$-conull subset $Z_W\subseteq X$ such that,  for all $z\in Z_W$,
$$
\psi\inv(W)_z=\Mgd{g\in G}{(g,z)\in \psi\inv(W)}=\psi_z\inv(W)
$$
is a $\lambda$-measurable subset of $G$. Thus, for all $z$ belonging to the $\mu$-conull set $Z=\bigcap_{W\in \ku W}Z_W$ and all $W\in \ku W$, $\psi_z\inv(W)$ is $\lambda$-measurable. So  $\psi_z$ is $\lambda$-measurable for all $z\in Z$. Because $\mu$ is fully supported, we see that $Z$ is dense in $X$ and the conclusion now follows from Theorem  \ref{thm:cocyclemain} (3).
\end{proof}


\section{Differentials and polynomial mappings}

\begin{proof}[Proof of Theorem \ref{thm:differentials}]
Assume $G\acts {}X$ is a transitive Polish group action and that $X\maps \phi H$ is a Baire measurable function with values in a Polish group $H$.

We first show that the differentials $G\times X\maps{d\phi}H$ and  $X\maps{d_g\phi}H$ are Baire measurable for all $g\in G$.  So let $g\in G$ be fixed and let $U\subseteq H$ be open. Because the group operation in $H$ is continuous, we have that 
$$
\Mgd{(h,k)\in H\times H}{ h\inv k\in  U}=\bigcup_nV_n\times W_n
$$
for some sequence of open sets $V_n,W_n\subseteq H$. If $G\times X\maps{\sf a} X$ denotes the  action map, it follows that
\maths{
\big(d_g\phi\big)\inv(U)
&=\bigcup_{n}\Mgd{x\in X}{\phi(gx)\in V_n \;\&\; \phi(x)\in W_n}\\
&=\bigcup_n \Big( \phi\inv(W_n)\cap g\inv\!\cdot \!\phi\inv(V_n)\Big)
}
and 
\maths{
\big(d\phi\big)\inv(U)
&=\bigcup_n\big\{(f,x)\in G\times X\del  fx\in \phi\inv (V_n)\;\&\;x\in \phi\inv(W_n)\big\}\\
&=\bigcup_n \Big({\sf a}\inv\big(\phi\inv (V_n)\big)\cap \big( G\times \phi\inv(W_n)\big)\Big).\\
}
Because $\phi$ is Baire measurable, the sets $\phi\inv (V_n)$ and $\phi\inv (W_n)$ have the Baire property. This shows that $d_g\phi$ is Baire measurable. Furthermore, because the action map $G\times X\maps{\sf a} X$  is surjective, continuous and open, the inverse image ${\sf a}\inv\big(\phi\inv (V_n)\big)$ will also have the Baire property, whereby $d\phi$ is Baire measurable.

Secondly, it is immediate that (1) implies (2), (3), (4) and (5).
Conversely, to see that (2) implies (1), suppose that $d\phi$ is continuuous and that $x_n\to x$. Since the action $G\curvearrowright X$ is transitive, the orbit map $g\in G\mapsto gx\in X$ is open by Effros' theorem \cite[Theorem 2.1]{effros} and so there is a sequence $(g_n)$ in $G$ converging to $1$ with $x_n=g_nx$. So, by the continuity of $d\phi$, we have that
$$
\lim_n \phi(x_n)\inv\phi(x)
=\lim_n \phi(g_nx)\inv\phi(x)
=\lim_nd\phi(g_n,x)
=d\phi(1,x)=1,
$$
whence $\phi(x_n)\to \phi(x)$, showing continuity of $\phi$.

Observe also that, because  $\big(d\phi\big)_g=d_g\phi$, each of (3) and (4) imply that the set
$$
\Sigma=\Mgd{g\in G}{(d\phi)_g \text{ is continuous\,}}
$$
must generate $G$ and hence, by Remark \ref{remark} and Theorem \ref{thm:cocyclemain}(1), that $d\phi$ is continuous. So (3) and (4) each imply (2).

To see that (5) implies (1), assume that (1) fails. Then, by the implication (4)$\saa$(1), there is some $g_1\in G$ so that also $d_{g_1}\phi$ is discontinuous and still is Baire measurable. Continuing like this, we may in fact choose an infinite sequence $g_1,g_2,g_3,\cdots$  so that $d_{g_k}\cdots d_{g_1}\phi$ is discontinuous for all $k\geqslant 0$, whereby also (5) fails.
\end{proof}


\begin{bibdiv}
\begin{biblist}



\bib{becker}{article}{
   author={Becker, Howard},
   title={Cocycles and continuity},
   journal={Trans. Amer. Math. Soc.},
   volume={365},
   date={2013},
   number={2},
   pages={671--719},
   issn={0002-9947},
}











\bib{effros}{article}{
   author={Effros, Edward G.},
   title={Transformation groups and $C\sp{\ast} $-algebras},
   journal={Ann. of Math. (2)},
   volume={81},
   date={1965},
   pages={38--55},
   issn={0003-486X},
}













\bib{kechris}{book}{
   author={Kechris, Alexander S.},
   title={Classical descriptive set theory},
   series={Graduate Texts in Mathematics},
   volume={156},
   publisher={Springer-Verlag, New York},
   date={1995},
   pages={xviii+402},
   isbn={0-387-94374-9},
}


\bib{leibman}{article}{
   author={Leibman, A.},
   title={Polynomial mappings of groups},
   journal={Israel J. Math.},
   volume={129},
   date={2002},
   pages={29--60},
   issn={0021-2172},
}


\bibitem{meyerovitch} T. Meyerovitch and  O. N. Solan, {\em Automatic continuity of polynomial maps and cocycles}, 
to appear in ``Proceedings of the American Mathematical Society.''






\bib{pettis}{article}{
   author={Pettis, B. J.},
   title={On continuity and openness of homomorphisms in topological groups},
   journal={Ann. of Math. (2)},
   volume={52},
   date={1950},
   pages={293--308},
   issn={0003-486X},
 }




\bib{autom}{article}{
   author={Rosendal, Christian},
   title={Automatic continuity of group homomorphisms},
   journal={Bull. Symbolic Logic},
   volume={15},
   date={2009},
   number={2},
   pages={184--214},
   issn={1079-8986},
}

















\end{biblist}
\end{bibdiv}
\end{document}